%% file: doc.tex
\documentclass[preprint,12pt]{elsarticle}

\usepackage[utf8x]{inputenc}
\usepackage[margin=3cm]{geometry}
\usepackage{verbatim}
\usepackage{color}
\usepackage[table]{xcolor}
\usepackage{listings}
\usepackage{amssymb,amsfonts,amsthm,amsmath}
\usepackage{multind}
\usepackage{url}
\usepackage{enumerate}
\usepackage{todonotes}
\usepackage[all]{xy}
\usepackage{multirow}
\usepackage{attachfile}
\usepackage{stmaryrd}
\usepackage{footnote}
\makesavenoteenv{tabular}
\usepackage{hyperref}

\makeatletter
\newcommand\footnoteref[1]{\protected@xdef\@thefnmark{\ref{#1}}\@footnotemark}
\makeatother

\hypersetup{
	pdfborder={0 0 0}
}

\definecolor{grey}{rgb}{0.95,0.95,0.95}
\definecolor{green}{rgb}{0.2,0.6,0.4}

\newcommand{\imp}{\rightarrow}

\newcommand{\Pb}{\mathbb{P}}
\newcommand{\Qb}{\mathbb{Q}}

\newcommand{\Psf}{\mathsf{P}}
\newcommand{\Qsf}{\mathsf{Q}}

\newcommand{\Gcal}{\mathcal{G}}

\newcommand{\Ical}{\mathcal{I}}

\newcommand{\Mcal}{\mathcal{M}}

\newcommand{\Rcal}{\mathcal{R}}
\newcommand{\Scal}{\mathcal{S}}
\newcommand{\Kcal}{\mathcal{K}}
\newcommand{\Tcal}{\mathcal{T}}
\newcommand{\Wcal}{\mathcal{W}}

\renewcommand{\setminus}{\smallsetminus}

\newcommand{\Odd}{\mathrm{Odd}}

\newcommand{\set}[1]{\left\{ #1 \right\}}
\newcommand{\card}[1]{\left| #1 \right|}

\newcommand{\s}[1]{\ensuremath{\sf{#1}}}

\newcommand{\dnrs}[1]{#1\mbox{-}\s{DNR}}
\DeclareMathOperator{\rca}{\s{RCA}_0}

\DeclareMathOperator{\wkl}{\s{WKL}_0}
\DeclareMathOperator{\wwkl}{\s{WWKL}_0}
\DeclareMathOperator{\dnr}{\s{DNR}}

\DeclareMathOperator{\rkl}{\s{RKL}}
\DeclareMathOperator{\rwkl}{\s{RWKL}}
\DeclareMathOperator{\rwwkl}{\s{RWWKL}}

\DeclareMathOperator{\rt}{\s{RT}}
\DeclareMathOperator{\srt}{\s{SRT}}

\DeclareMathOperator{\rrt}{\s{RRT}}

\DeclareMathOperator{\ads}{\s{ADS}}

\DeclareMathOperator{\scac}{\s{SCAC}}

\DeclareMathOperator{\rcolor}{\s{RCOLOR}}

\DeclareMathOperator{\emo}{\s{EM}}

\newtheoremstyle{custom}
  {10pt}
  {10pt}
  {\normalfont}
  {}
  {\bfseries}
  {}
  { }
  {}

\theoremstyle{custom}

\usepackage{xcolor}	
\usepackage{soul}

\newtheorem{theorem}{Theorem}[section]
\newtheorem{lemma}[theorem]{Lemma}
\newtheorem{definition}[theorem]{Definition}

\newtheorem{example}[theorem]{Example}

\journal{Information Processing Letters}

\begin{document}

\begin{frontmatter}

\title{Ramsey-type graph coloring and diagonal non-computability}
\author{
  Ludovic Patey
}

\begin{abstract}
A function is diagonally non-computable (d.n.c.) if it diagonalizes against 
the universal partial computable function. 
D.n.c.\ functions play a central role in algorithmic randomness and reverse mathematics.
Flood and Towsner asked for which functions~$h$, the principle stating
the existence of an~$h$-bounded d.n.c.\ function ($\dnrs{h}$) implies the Ramsey-type König's lemma ($\rwkl$).
In this paper, we prove that for every computable order~$h$,
there exists an~$\omega$-model of~$\dnrs{h}$ which is not a not model of the Ramsey-type graph coloring principle
for two colors~($\rcolor_2$) and therefore not a model of~$\rwkl$.
The proof combines bushy tree forcing and a technique introduced by Lerman, Solomon and Towsner
to transform a computable non-reducibility into a separation over~$\omega$-models.
\end{abstract}


\end{frontmatter}

\section{Introduction}

\input{parts/part0-introduction}

\section{Bushy tree forcing}\label{sect:bushy-tree}
\input{parts/part1-bushy-trees}

\section{Separation over~$\omega$-models}\label{sect:separation-omega}
\input{parts/part2-omega-models}

\section{Main result}
\input{parts/part3-main-result}

\vspace{0.5cm}

\noindent \textbf{Acknowledgements}. The author is grateful to Manuel Lerman, Reed Solomon
and Henry Towsner for kindly answering the author's questions about their forcing techniques.
The author is also thankful to his PhD advisor Laurent Bienvenu for useful comments and discussions.

\end{document}

%% file: parts/part0-introduction.tex
In this paper, we deepen our understanding of the relation between
diagonal non-computability and consequences of the Ramsey-type weak König's lemma
by separating related principles within the framework of reverse mathematics.

Diagonal non-computability plays an important role in algorithmic randomness
in which is it known to computationally coincide with infinite subsets
of Martin-Löf randoms. In a reverse mathematical perspective, the corresponding
principle has been used as a discriminating statement for separating
theorems. Some principles in reverse mathematics -- $\wkl$~\cite{Jockusch1972Degrees}, 
$\rrt^2_2$~\cite{MillerPersonal}, $\rwwkl$~\cite{Bienvenu2014Ramsey}, can be characterized by d.n.c-like statements.
Therefore, diagonal non-computability can be considered as a unifying framework
for comparing existing principles.

The study of Ramsey-type versions of principles which are not consequences
of the Ramsey theorem has been initiated by Flood~\cite{Flood2012Reverse} with his 
Ramsey-type weak König's lemma ($\rwkl$). As pointed out by Flood \& Towsner~\cite{Flood2014Separating}, this principle aims to play
a central role in reverse mathematics, as the ``missing link'' explaining
the relation between Ramsey's theorem for pairs~($\rt^2_2$) 
and the weak König's lemma~($\wkl$). Indeed, previous proofs of~$\rt^2_2$ seemed
to require $\wkl$ even though Liu~\cite{Liu2012RT22} proved that~$\wkl$ is not a consequence of~$\rt^2_2$.
$\rwkl$ contains in fact the exact combinatorics needed in the proofs of~$\rt^2_2$
or even weaker statements like the Erd\H{o}s-Moser theorem~\cite{Flood2014Separating,Bienvenu2014Ramsey}.

\subsection{Diagonally non-computable functions}

A function~$f$ is \emph{diagonally non-computable} (d.n.c.) relative to~$X$ if for every index~$e$, $f(e) \neq \Phi^X_e(e)$,
where~$\Phi_e$ is an effective enumeration of all Turing machines.
D.n.c.\ functions have been extensively studied in literature~
\cite{Ambos-Spies2004Comparing,Dorais2014Comparing,Jockusch2013Diagonally,kumabe2009fixed}. The degrees of d.n.c. functions
coincide with the degrees of fixed-point free functions~\cite{Jockusch1989Recursively} 
and the degrees of infinite subsets of Martin-Löf random reals~
\cite{Greenberg2009Lowness,Kjos-Hanssen2009Infinite,Kuvcera1985Measure}.

\begin{definition}[Diagonal non-computability]
A function~$f$ is \emph{$h$-bounded} for some function~$h$ if $f(x) \leq h(x)$ for every input~$x$.
$\dnr$ is the statement ``For every definable set~$X$, there exists a function d.n.c.\ relative to~$X$''
and for every function~$h$, $\dnrs{h}$ is the statement ``For every definable set~$X$, there
exists an~$h$-bounded function d.n.c.\ relative to~$X$''.
\end{definition}

In particular, when $h$ is the function constantly equal to~$k$, we say that $f$ is $k$-bounded.
Friedberg~\cite{JockuschJr1989Degrees} proved that every $k$-bounded d.n.c.\ function computes a 2-bounded d.n.c. function.
Jockusch~\cite{JockuschJr1989Degrees}
proved that this reduction is not uniform, and Dorais \& Shafer~\cite{Dorais2014Comparing} constructed a non-standard structure
satisfying the statement $(\exists k)\dnrs{k}$ and which does not
contain any 2-bounded d.n.c.\ function.
The degrees of~2-bounded d.n.c.\ functions coincide with the degrees of completion
of Peano arithmetic~\cite{Jockusch1972Degrees}.

When considering orders $h$, that is, non-decreasing and unbounded functions, $h$-bounded d.n.c.\ functions are known to form a strict hierarchy within reverse mathematics~\cite{Ambos-Spies2004Comparing,Flood2014Separating}.
Recently, Bienvenu and the author~\cite{BienvenuEvery} proved that witnesses to the strictness of this hierarchy
can be constructed by probabilistic means. Jockusch~\cite{JockuschJr1989Degrees} showed the existence of computable orders~$h$
such that every Martin-Löf random computes an~$h$-bounded d.n.c.\ function.

\subsection{Ramsey-type principles}

Most of principles in reverse mathematics are of the form
$$
(\forall X)[\Phi(X) \imp (\exists Y)\Psi(X, Y)]
$$
where $\Phi$ and $\Psi$ are arithmetic formulas. A set $X$ such that $\Phi(X)$ holds is called an \emph{instance}
and a set~$Y$ such that $\Psi(X, Y)$ holds is a \emph{solution} to~$X$.
The Ramsey-type version of such a principle informally consists in asserting, for every instance~$X$,
the existence of infinitely many bits of information ``compatible'' with a solution to~$X$.

Among the principles in reverse mathematics, the weak König's lemma ($\wkl$) states,
for every infinite binary tree~$T$, the existence of an infinite path through~$T$. 
In this case, its Ramsey-type version -- i.e. the Ramsey-type König's lemma ($\rwkl$) -- states the existence
of a direction~$d$ -- left or right -- and an infinite set of depths~$H$, 
such that there are arbitrarily large nodes in~$T$ which go on direction~$d$ at each depth of~$H$.
We now give a precise definition of~$\rwkl$. We denote by~$\omega$ the set of positive integers
and by~$2^{<\omega}$ the set of finite binary strings.

\begin{definition}[Ramsey-type weak König's lemma]{\ }
A set $H \subseteq \omega$ is \emph{homogeneous} for a $\sigma \in 2^{<\omega}$
if $(\exists c < 2)(\forall i \in H)[i < |\sigma| \imp \sigma(i) = c]$.
A set~$H$ is \emph{homogeneous} for an infinite tree $T \subseteq 2^{<\omega}$ if the tree $\{\sigma \in T : H \mbox{ is homogeneous for } \sigma \}$ is infinite. $\rwkl$ is the statement ``Every infinite subtree of $2^{<\omega}$ has an infinite homogeneous set''.
\end{definition}

Flood introduced the principle in~\cite{Flood2012Reverse} under the name~$\rkl$, 
and proved that it is a strict consequence of the stable Ramsey's theorem for pairs ($\srt^2_2$)
and~$\wkl$ over~$\rca$. He also proved that~$\rca \vdash \rwkl \imp \dnr$ and asked whether the implication is strict.
Bienvenu \& al.\ \cite{Bienvenu2014Ramsey} studied extensively variants of~$\rwkl$ and constructed
an $\omega$-model of~$\dnr$ -- and even of the weak weak König's lemma ($\wwkl$) --  which is not a model of~$\rwkl$.
Flood \& Towsner~\cite{Flood2014Separating} reproved the existence of an~$\omega$-model of~$\dnr$ which is not a model of~$\rwkl$
using the techniques developped by Lerman~\& al.\ in~\cite{Lerman2013Separating}. They asked in particular for which functions~$h$
the principle~$\dnrs{h}$ implies~$\rwkl$. In this paper, we answer this question by proving that for every computable order~$h$,
$\dnrs{h}$ does not even imply a weaker statement over~$\omega$-models.
Note that this result is incomparable (neither stronger nor weaker) than the result of Bienvenu \& al.~\cite{Bienvenu2014Ramsey}
since there exists computable orders~$h$ such that the measure of oracles computing a $h$-bounded d.n.c.\ function
is null~\cite{MushfeqKhan2014Forcing}.

\subsection{Ramsey-type graph coloring}

A graph $G = (V,E)$ is \emph{$k$-colorable} if there is a function $f \colon V \imp k$ such that $(\forall x,y \in V)(\{x,y\} \in E \imp f(x) \neq f(y))$, and a graph is \emph{locally $k$-colorable} if every finite subgraph is $k$-colorable.
The statement that every locally $k$-colorable graph admits an infinite~$k$-coloration for some fixed~$k \geq 2$
has been proved to be equivalent to~$\wkl$ over~$\rca$ by Hirst in~\cite{Hirst1990Marriage}.
We focus on the Ramsey-type version of this principle.

\begin{definition}[Ramsey-type graph coloring]{\ }
Let $G = (V,E)$ be a graph.  A set $H \subseteq V$ is {\itshape $k$-homogeneous for $G$} if every finite $V_0 \subseteq V$ induces a subgraph that is $k$-colorable by a coloring that colors every vertex in $V_0 \cap H$ color~$0$.
$\rcolor_k$ is the statement  ``for every infinite, locally $k$-colorable graph $G = (V,E)$, there is an infinite $H \subseteq V$ that is $k$-homogeneous for $G$.''
\end{definition}

The principles~$\rcolor_k$ have been introduced by Bienvenu \& al. in~\cite{Bienvenu2014Ramsey}.
They proved that $\rwkl$ is equivalent to~$\rcolor_k$ for every~$k \geq 3$.
and constructed an~$\omega$-model of~$\wwkl$ (and in particular of~$\dnr$) which is not a model of~$\rcolor_2$.
It is currently unknown whether~$\rcolor_2$ is strictly weaker than~$\rwkl$ and in particular
whether or not it implies~$\dnr$ over~$\rca$. The principle $\rcolor_2$ seems easier to manipulate than~$\rwkl$,
and most proofs of the type ``$\Phi$ does not prove~$\rwkl$'' can be strengthened 
without any additional effort to ``$\Phi$ does not prove~$\rcolor_2$''.

\subsection{Notations}

\emph{String, sequence}.
Fix a function $h : \omega \to \omega$.
A \emph{string} (over $h$) is an ordered tuple of integers $a_0, \dots, a_{n-1}$
(such that $a_i < h(i)$ for every $i < n$). A \emph{sequence}  (over $h$)
is an infinite listing of integers $a_0, a_1, \dots$ (such that $a_i < h(i)$ for every $i \in \omega$).
We denote by $\preceq$ the prefix relation between two strings or between a string and a sequence.
For $s \in \omega$,
$h^s$ is the set of strings of length $s$ over~$h$,
$h^{<s}$ is the set of strings of length strictly smaller than~$s$ over~$h$,
$h^{<\omega}$ is the set of finite strings over~$h$
and $h^{\omega}$ is the set of sequences (i.e. infinite strings)
over~$h$. When $h$ is the constant function $k$, we write
$k^s$ (resp. $k^{<s}, \dots$) for $h^s$ (resp. $h^{<s}, \dots$).
Given a string $\sigma \in h^{<\omega}$, we denote by $|\sigma|$ its length.
Given two strings $\sigma, \tau \in h^{<\omega}$, $\sigma$ is a \emph{prefix}
of $\tau$ (written $\sigma \preceq \tau$ if there exists a string $\rho \in h^{<\omega}$
such that $\sigma \rho = \tau$.

\emph{Tree, path}.
A tree $T \subseteq k^{<\omega}$ is a set downward closed by prefix relation.
A \emph{binary} tree is a tree~$T \subseteq 2^{<\omega}$.
A set $P \subseteq \omega$ is a \emph{path} though~$T$ if for every $\sigma \prec P$,
$\sigma \in T$. A string $\sigma \in k^{<\omega}$ is a \emph{stem} of a tree $T$
if every $\tau \in T$ is comparable with~$\sigma$.
Given a tree $T$ and a string $\sigma \in T$,
we denote by $T^{[\sigma]}$ the subtree $\{\tau \in T : \tau \preceq \sigma \vee \tau \succeq \sigma\}$.

\emph{Turing functionals}.
We fix an effective enumeration~$\Phi_0, \Phi_1, \dots$ of all Turing functionals.
Given an oracle~$X$, we write $W^X_e = dom(\Phi_e^X)$, that is,
$W^X_e$ is the set of all inputs~$x$ such that $\Phi_e^X(x) \downarrow$.
We may apply a Turing functional over a string~$\sigma \in \omega^{<\omega}$.
In this case, by convention, $\Phi_e^\sigma(x) \downarrow \imp x < |\sigma|$.

An \emph{order}~$h$ is an unbounded, non-decreasing function over integers.
Given two sets~$A$ and~$B$, we write~$A \subseteq_{fin} B$ to mean that $A$ is a finite subset of~$B$.
We also denote by $A \times B$ the set of unordered pairs $\{a, b\}$ where $a \in A$ and $b \in B$.
Last, the notation~$A > x$ means that the set~$A$ is empty or its least element is greater than~$x$.

%% file: parts/part1-bushy-trees.tex
The specificity of the combinatorics of diagonally non-computable functions~$f$ is 
that $f$ has to avoid \emph{at most one} bad value on each input~$e$.
Therefore, if we consider two functions~$f_0$ and $f_1$ having two different values
on input~$e$, at least one of them diagonalize against~$\Phi_e(e)$.
The exploitation of this observation leads to the notion of bushy tree forcing,
in which the construction of a d.n.c.\ function is done by simply applying cardinality arguments.
Bushy tree forcing was first developed in~\cite{Ambos-Spies2004Comparing}. Since then, it has been successfully applied
to various problematics, like proving the existence of a d.n.c.\ function of minimal degree~\cite{kumabe2009fixed}.
We state the core lemmas of bushy tree forcing without proving them. See the survey~\cite{MushfeqKhan2014Forcing}
for detailed proofs.

\begin{definition}[Bushy tree]
Fix an integer~$k$ and a string $\sigma \in \omega^{<\omega}$.
A tree $T$ is \emph{$k$-bushy above $\sigma$} if it has stem~$\sigma$
and whenever $\tau \succeq \sigma$ is not a leaf of~$T$,
it has at least~$k$ immediate children.
\end{definition}

\begin{definition}[Big set, small set]
Fix an integer~$k$ and some string $\sigma \in \omega^{<\omega}$.
A set $B \subseteq \omega^{<\omega}$ is \emph{$k$-big above $\sigma$} 
if there exists a finite tree $T$ $k$-bushy above $\sigma$ 
such that all leafs of~$T$ are in~$B$. 
If no such tree exists, $B$ is said to be \emph{$k$-small above~$\sigma$}.
\end{definition}

Those notions of~$k$-bigness and $k$-smallness can be justified as follows:
Suppose we have constructed an initial segment~$\sigma$ of a d.n.c.\ function~$f$.
We want to decide whether a $\Sigma^0_1$ property~$\varphi(f)$ will hold or not.
A property~$\varphi$ induces a set of strings
$B = \{ \tau \in \omega^{<\omega} : \varphi(\tau) \mbox{ holds}\}$.
Deciding~$\varphi(f)$ consists in either finding an extension~$\tau \succeq \sigma$
which is in~$B$ and such that~$\tau$ is again an initial segment of a d.n.c.\ function,
or ensuring that no extension of~$\sigma$ will meet~$B$.

Suppose that~$B$ is $2$-big above~$\sigma$. By definition, there exists
a finite tree~$T$ $2$-bushy above~$\sigma$ and whose leafs are all in~$B$.
In other words, for each leaf~$\tau$ in~$T$, $\varphi(\tau)$ holds.
By our previous observation, at each depth, there exists at most one bad value.
Therefore the tree~$T$ must have one leaf~$\tau$ which is an initial segment of a d.n.c.\ function.
By hypothesis, $\varphi(\tau)$ holds and $\sigma$ being the stem of~$T$, $\tau \succeq \sigma$.
We can then take this~$\tau$ as our desired extension and continue the construction.

Suppose now that~$B$ is $2$-small above~$\sigma$.
Informally, it means that~$\sigma$ does not have enough extensions satisfying~$\varphi$.
In this case, we want to ensure that no further extension of~$\sigma$
will meet~$B$ while continuing the construction. 
But, how can we ensure there will always be a d.n.c.\ extension avoiding~$B$
when choosing a leaf in~$T$ while considering other properties~$\varphi$ ?
It suffices to slightly modify our question for the next property~$\varphi_1$
and ask whether the corresponding set~$B_1$ is~$3$-big, and so on.

At this stage, it becomes clear that the good forcing notion is
a pair $(\sigma, B)$ where $\sigma \in \omega^{<\omega}$ and $B$
is a set which is $k$-small above~$\sigma$ for some~$k$.
The set $B$ is intuitively the ``bad'' set of extensions we want to avoid,
and therefore may only increase.
Another condition $(\tau, C)$ extends $(\sigma, B)$
if $\sigma \preceq \tau$ and $B \subseteq C$.

What about the constraint that~$\sigma$ must be an initial segment
of a d.n.c.\ function ?
Fix any set $X$, and consider the set $B^X_{DNC}$ of all strings which are not initial segments 
of any function d.n.c.\ relative to~$X$:
\[
B^X_{DNC} = \set{\sigma \in \omega^{<\omega} : 
  (\exists e < \card{\sigma}) \Phi^{X}_e(e) \downarrow = \sigma(e) }
\]
One can easily see that $B^X_{DNC}$ is $2$-small above the empty string.
Therefore, by starting the construction with the initial condition $(\epsilon, B^X_{DNC})$,
we ensure that the resulting function will be~d.n.c.\ relative to~$X$.
Notice that the choice of the set~$X$ has no impact on the construction as we only care about the bushiness
of the corresponding set~$B^X_{DNC}$.

The following three lemmas are at the core of every bushy tree argument.

\begin{lemma}[Concatenation]\label{lem:concatenation-prop}
Fix an integer~$k$.
Suppose that $A \subseteq \omega^{<\omega}$ is $k$-big above~$\sigma$.
If $A_\tau \subseteq \omega^{<\omega}$ is $k$-big above $\tau$
for every $\tau \in A$, then $\bigcup_{\tau \in A} A_\tau$ is $k$-big above~$\sigma$.
\end{lemma}

The concatenation property is very intuitive and is sufficient for
running a basic bushy tree argument.
Indeed, fix a condition~$(\sigma, B)$ where $B$ is $k$-bushy above~$\sigma$.
We need to decide whether a property~$\varphi$ 
with corresponding set~$C = \{\tau \in \omega^{<\omega} : \varphi(\tau) \mbox{ holds}\}$ can be forced.
The question becomes ``Is the set~$B \cup C$ $k$-big above~$\sigma$ ?''.
In the case of a positive answer, let~$T$ be a finite tree $k$-bushy above~$\sigma$
whose leaves are in~$B \cup C$. By the concatenation property, $B$
must be~$k$-small above one of the leaves~$\tau$ in~$T$, otherwise $B$ would be $k$-big above~$T$.
The condition~$(\tau, B)$ is the desired extension forcing~$\varphi(f)$ to hold.
In the case of a negative answer, $(\sigma, B \cup C)$ is an extension forcing~$\varphi(f)$ not to hold.

\begin{lemma}[Smallness additivity] \label{lem:smallness-add}
Suppose that $B_1, B_2, \ldots, B_n$ are subsets of $\omega^{<\omega}$, $k_1$, $k_2$, ..., $k_n$ are integers, and $\sigma \in \omega^{<\omega}$.
If $B_i$ is $k_i$-small above~$\sigma$ for all~$i$, then $\bigcup_i B_i$ is $(\sum_i k_i)$-small above $\sigma$. 
\end{lemma}

Smallness additivity has a significant impact on the effectiveness of the forcing argument.
Note that we did not impose effectiveness constraints on the set~$B$ in the definition of a condition~$(\sigma, B)$,
and indeed, sets may have arbitrary complexity, like~$B^X_{DNC}$ which is strictly~$X$-c.e.
The question of the $k$-bigness of a set~$B$ is $B$-c.e.
However, the considered properties have often a c.e.\ corresponding set~$C$.
In this case, given a condition~$(\sigma, B)$, the question of~$k$-bigness of~$B \cup C$
is~$\Sigma^{0, B}_1$ whereas the $k$-bigness of~$C$ is~$\Sigma^0_1$.

Smallness additivity enables us to ask the question of~$k$-bigness of~$C$ independently of~$B$, and
combine the two sets after. Indeed, suppose that~$B$ is $k$-small above~$\sigma$,
and ask whether~$C$ is $k$-big above~$\sigma$.
If so, by the usual reasoning, find a leaf~$\tau$ in the bushy tree witnessing $k$-bigness of~$C$,
such that $B$ is still $k$-small above~$\tau$. In the other case,
$(\sigma, B \cup C)$ is a valid extension as $B \cup C$ is $2k$-small above~$\sigma$.

The effectiveness issue will be particularly important when defining the computable
graph in the ground construction, so that it satisfies specific properties
necessary to the iteration forcing.

\begin{lemma}[Small set closure]\label{lem:small-set-closure}
We say that $B  \subseteq \omega^{<\omega}$ is \emph{$k$-closed} if whenever $B$ is $k$-big above a string $\rho$ then $\rho \in B$. Accordingly, the \emph{$k$-closure} of any set~$B \subseteq \omega^{<\omega}$ is the set $C = \set{\tau \in \omega^{<\omega} : B \mbox{ is $k$-big above } \tau}$. If $B$ is $k$-small above a string~$\sigma$, then its closure is also $k$-small above $\sigma$.
\end{lemma}

Considering the~$k$-closure~$C$ of the set~$B$ instead of the set~$B$ itself in the condition~$(\sigma, B)$
simplifies the reasoning. Indeed, taking an extension~$\tau$ of~$\sigma$ avoiding the set~$C$
is enough to deduce by~$k$-closure that $C$ is $k$-small above~$\tau$ and therefore that~$(\tau, C)$ is a valid condition.

%% file: parts/part2-omega-models.tex
An \emph{$\omega$-structure} is a structure $\Mcal = (\omega, S, +, \cdot, <)$
where $\omega$ is the set of standard integers, $+$, $\cdot$ and $<$ 
are the standard operations over integers and $S$ is a set of reals
such that~$\Mcal$ satisfies the axioms of~$\rca$. Notice that $\Mcal$
is fully specified by its second-order part~$S$, and we may therefore amalgamate~$\Mcal$ with~$S$.
Friedman~\cite{Friedman1974Some} characterized the second-order parts~$S$ of~$\omega$-structures
as those forming a \emph{Turing ideal}, that is, a set of reals closed by Turing join
and downward-closed by Turing reduction.

Separation of a principle~$\Psf$ from a principle~$\Qsf$ over $\omega$-models differs 
from computable non-reducibility  in that solutions to~$\Psf$-instances
may enable to define other~$\Psf$-instances whose solutions compute solutions to~$\Qsf$-instances.
Therefore, we need to be careful about iterated applications of principles.
In their paper~\cite{Lerman2013Separating}, Lerman, Solomon \& Towsner
introduced a general technique for transforming a one-step separation  -- i.e. proof
of computable non-reducibility -- into a separation over~$\omega$-models, 
and used it for separating~$\emo$ from $\srt^2_2$ and $\ads$ from~$\scac$ over~$\omega$-models.
We briefly present their method.

Fix two principles~$\Psf$ and $\Qsf$, each of the form~$(\forall X)[\Phi(X) \imp (\exists Y)\Psi(X, Y)]$
where $\Phi$ and $\Psi$ are arithmetical formulas.
The construction of an~$\omega$-model of~$\Psf$ which is not a model of~$\Qsf$
consists in creating a Turing ideal~$\Ical$ together with a fixed $\Qsf$-instance $I_0 \in \Ical$,
such that every $\Psf$-instance $J \in \Ical$ has a solution in~$\Ical$,
whereas $I_0$ contains no solution in~$\Ical$. The techniques introduced by Lerman \& al.\ 
include

\begin{itemize}
	\item A \emph{ground construction}, whose goal is to define
	a fixed instance~$I_0$ of~$\Qsf$ with no computable solution, and ensuring
	some fairness property about the iteration forcing. Lerman \& al. used the forcing
 	framework for creating their instances although computable instances
	could be defined. The constructed instance is considered as the \emph{ground context}.

	\item An \emph{iteration forcing} which, given a context~$X$ sharing the same
	properties with the ground context -- i.e. $I_0$ has no $X$-computable solution
	and the iteration forcing relative to~$X$ satisfies some fairness properties --
	and an $X$-computable $\Psf$-instance $J$, constructs a solution $G$ to~$J$
	such that~$X \oplus G$ satisfies again the contextual properties.
\end{itemize}

Choosing carefully the enumeration $J_0, J_1, \dots$ of all~$\Psf$-instance functionals,
one can define an infinite sequence of contexts~$X_0 = I_0 \leq_T X_1 \leq_T X_2 \dots$
such that
\begin{itemize}
	\item $X_i$ does not compute a solution to~$I_0$
	\item $X_{i+1}$ computes a solution to~$J_i^{X_i}$
\end{itemize}
and such that the Turing ideal generated by the~$X_i$'s is
the second-order part of an~$\omega$-model of~$\Psf$.

The specificity of the techniques of Lerman \& al.\ is that instead
of creating a $\Qsf$-instance~$I_0$ which will try to satisfy the least
requirements necessary to the iteration forcing at each level,
$I_0$ considers only the ground context, and satisfies a much
more general notion of requirement. This notion has to be general
enough so that we can express the requirements at level~$i+1$
inside the requirements at level~$i$,
and therefore satisfy transparently requirements at all levels.

Given a~$\Qsf$-instance $I_0$, a context~$X$ and an iteration forcing notion~$\Pb^X$,
they define their general requirement~$\Kcal^{X, I_0}$ to be an
upward-closed set of forcing conditions~$c \in \Psf^X$ with a particular shape
$$
\Kcal^{X, I_0} = \{ c \in \Pb^X : (\exists F \subseteq_{fin} I_0)R^X_\Kcal(c, F)\}
$$
where $R^X_\Kcal$ is an $X$-computable predicate with some additional properties
specific to the forcing notion we consider. The set $F$ has to be understood as a finite sub-instance of~$I_0$. 
If $G(c)$ is the initial segment of the generic solution~$G$ we are constructing, an intended particular 
requirement is
$$
\Wcal_e^{X, I_0} = \{ c \in \Pb^X : (\exists F \subseteq_{fin} I_0) \Phi_e^{X \oplus G(c)} \mbox{ is not an $F$-solution} \}
$$
where ``$\Phi_e^{X \oplus G(c)}$ is not an $I_0$-solution'' has to be understood as
``$\Phi_e^{X \oplus G(c)}$ halts on an initial segment which cannot be extended into an~$F$-solution.''
Therefore, if some condition~$c \in \Wcal_e$, then it cannot be extended into an infinite set~$G$
such that~$\Phi^{X \oplus G}_e$ is a solution to~$I_0$.

Of course, we cannot expect all our requirements~$\Wcal_e^{X, I_0}$ to be~$\Pb^X$-dense,
as there are Turing indices~$e$ such that $\Phi_e^{X \oplus G}$ is nowhere defined, whatever the oracle~$G$.
In those cases, one may want to force~$\Phi_e^{X \oplus G}$ not to be total.
In Lerman \& al.\ framework, this is done by ensuring that further extensions avoid the following set for some~$x \in \omega$:
$$
\Kcal^{X, [x, +\infty)} = \{ c \in \Pb^X : (\exists F \mbox{ finite $\Qsf$-instance over } [x, +\infty)) R^X_\Kcal(c, F)\}
$$
In the case of~$\Wcal_e^{X, [x, +\infty)}$, it corresponds to forcing $\Phi_e^{X \oplus G}$ not to halt on values greater than~$x$.
The fairness condition that the $\Qsf$-instance~$I_0$ has to satisfy states that if there are valid extensions
in $\Kcal^{X, F}$, for finite $\Qsf$-instance~$F$ over an arbitrarily far domain, then $I_0$ will
take one of those~$F$ as a part of its own instance, and therefore there will exist a valid extension
in~$\Kcal^{X, I_0}$. The combinatorics of~$\Psf$ have to be good enough to ensure that,
in a context~$X$ where the fairness condition holds, either we will find an extension
in~$\Kcal^{X, I_0}$, or we will avoid the requirement~$\Kcal^{X, [x, +\infty)}$ for some~$x \in \omega$.

%% file: parts/part3-main-result.tex
Flood and Towsner~\cite{Flood2014Separating} asked for which functions~$h$ the statement $\dnrs{h}$ implies~$\rwkl$ over~$\rca$.
In the case of functions whose range is bounded by some constant~$k$, $\dnrs{h}$ implies
$\dnrs{k}$ and therefore is equivalent to~$\wkl$. We prove that in the case of computable orders~$h$,
$\dnrs{h}$ does not even imply the Ramsey-type graph coloring principle for pairs over~$\omega$-models.

\begin{theorem}\label{thm:computable-h-rcolor2}
For every computable order~$h$, there exists an~$\omega$-model of~$\dnrs{h}$ which is not a model of~$\rcolor_2$.
\end{theorem}

The proof of Theorem~\ref{thm:computable-h-rcolor2} follows the pattern presented in section~\ref{sect:separation-omega},
in which the iteration forcing is a variant of the bushy tree forcing presented in section~\ref{sect:bushy-tree}.
Fix a computable order~$h$. The proof is divided into two main parts:
\begin{itemize}
	\item The first part describes the iteration forcing in a context~$X$, assuming
	that we have already constructed a computable, locally 2-colorable graph~$\Gcal$
	ensuring that the context~$X$ satisfies some properties.

	\item The second part is the actual finite-injury priority construction of the graph~$\Gcal$
	so that it satisfies the required properties over the ground context
	(the latter being the empty context as~$\Gcal$ is a computable graph).
\end{itemize}

Finally, we describe the construction of the~$\omega$-model
of~$\dnrs{h}$ which is not a model of~$\rcolor_2$.

\subsection{Iteration forcing}

The general context of the iteration forcing is a fixed set $X$ such that
\begin{itemize}
	\item $X$ does not compute a solution to~$\Gcal$
	\item each requirement $\Kcal^{X, \Gcal}$ is uniformly dense (in a sense defined below)
\end{itemize}

Our goal is to define an $h$-bounded d.n.c.\ function $f$ such that $X \oplus f$ satisfies the same
context properties, that is~$X \oplus f$ does not compute a solution
to~$\Gcal$ and the requirements $\Kcal^{X \oplus f, \Gcal}$ are uniformly dense.

\begin{definition}
Our set of forcing conditions~$\Qb^X$ is the set of pairs $(\sigma, B)$ such that
$\sigma \in h^{<\omega}$ and $B \subseteq h^{<\omega}$ 
is a $h(|\sigma|)$-closed set of strings $h(|\sigma|)$-small above~$\sigma$.
We say that $(\tau, C) \leq (\sigma, B)$ if $\sigma \preceq \tau$ and $B \subseteq C$.
\end{definition}

Our initial condition will be $(\epsilon, B^X_{DNC})$. Therefore, every infinite
descending sequence of conditions will produce an $h$-bounded function d.n.c.\ relative to~$X$.

We say that a condition~$(\sigma, B)$ is \emph{$k$-roomy} if~$B$ is $k$-small above~$\sigma$
and $h(|\sigma|) \geq 4k$. The intuition of~$k$-roominess is that there is room for increasing
the size of the bad set~$B$. The choice of the coefficient is for the purpose of our combinatorics.
Notice that, thanks to the fact that $h$ is an order, every condition can be extended to a $k$-roomy
condition for some~$k \in \omega$.

\subsubsection{Requirements and uniform density}

As presented in section~\ref{sect:separation-omega},
a requirement is a set of conditions of the iteration forcing~$\Pb^X$
which enable to diagonalize positively against the instance~$\Gcal$ of~$\rcolor_2$.
In our case, a finite set of vertices~$F$ is extensible to an infinite solution
to~$\Gcal$ if and only if it does not contain a pair of vertices~$\{x, y\}$
which are the endpoints of a path of odd length in~$\Gcal$.
Based on this observation, we define our set of diagonalizing pairs as follows.

\begin{definition}
Given a locally 2-colorable graph~$\Gcal$, we denote by~$\Odd(\Gcal)$ the $\Gcal$-c.e.\
set of pairs~$\{x, y\}$ such that there exists a path of odd length 
whose endpoints are~$x$ and~$y$.
\end{definition}

We extend the notation~$\Odd(\Gcal)$ to sets of the form~$A_0 \times A_1$
where $A_0$ and $A_1$ are arbitrary sets of integers.
In this case, we can see~$A_0 \times A_1$ as the set of edges of a graph
(which is locally 2-colorable if and only if~$A_0 \cap A_1 = \emptyset$).
In particular, $\{x\} \in \Odd(A_0 \times A_1)$ for each~$x \in A_0 \cap A_1$.

A requirement is usually a set of conditions.
In our case, thanks to the smallness additivity property of the bushy tree
forcing, a requirement can be defined only in terms of the stem
part of the condition.

\begin{definition}\label{def:requirement}
A  \emph{requirement} is a set $\Kcal^{X, \Gcal}$ of finite strings over~$h$
which is closed under extensions and is defined by
$$
\Kcal^{X, \Gcal} = \{\tau \in h^{<\omega} : (\exists F \subseteq_{fin} \Odd(\Gcal)) R^X_\Kcal(\tau, F) \}
$$
for an $X$-computable relation $R^X_{\Kcal}$ such that if $R^X_\Kcal(\tau, F)$ holds for a finite set of pairs~$F$,
and $F_1$ is a finite set of singletons extensibles to pairs in~$F$ -- that is 
$(\forall \{z\} \in F_1)(\exists \{x,y\} \in F)[z \in \{x, y\}]$ -- then $R^X_\Kcal(\tau, F_1)$ also holds.
This last property will be used in the proof of Theorem~\ref{thm:uniform-to-settling}.
\end{definition}

We illustrate our general notion of requirement by describing
a particular scheme of requirements which will ensure that~$X \oplus f$
does not compute a solution to~$\Gcal$.
We shall see later another scheme of requirements which will
``propagate'' uniform density to the context~$X \oplus f$.

\begin{example}\label{ex:if-meets-wn-then-diag}
For each integer~$m$, define the requirement
$$
\Wcal_m^{X, \Gcal} = \{ \tau \in h^{<\omega} : (\exists \{a,b\} \in \Odd(\Gcal))
	\Phi_m^{X \oplus \tau}(a) = \Phi_m^{X \oplus \tau}(b) = 1 \}
$$

In this case, the~$X$-computable predicate~$R_{\Wcal_m}^X(\tau, F)$ is defined by
$$
F \neq \emptyset \wedge (\forall x \in \bigcup F)[\Phi_m^{X \oplus \tau}(x) = 1]
$$
Note that it satisfies the property of Definition~\ref{def:requirement}.
Suppose a condition $(\sigma, B)$ used to construct our generic function~$f$ satisfies~$\sigma \in \Wcal_m^{X, \Gcal}$.
Because~$\sigma$ is an initial segment of~$f$, we have successfully diagonalized 
against $\Phi_m^{X \oplus f}$ computing a solution to~$\Gcal$.
\end{example}

We can replace the graph~$\Gcal$ in a requirement~$\Kcal^{X, \Gcal}$
by a set of pairs~$A_0 \times A_1$, which can be thought of as the set of edges
of a finite or infinite graph, and consider the set
$$
\Kcal^{X, A_0 \times A_1} = \{ \tau \in h^{<\omega} : (\exists F \subseteq_{fin} \Odd(A_0 \times A_1)) R^X_\Kcal(\tau, F) \}
$$

\begin{definition}
We say $\Kcal^X$ is \emph{essential} below some $k$-roomy condition $(\sigma, B)$ if for every~$x$,
there exists a finite set~$A_0 > x$ such that for every~$y$, there exists a finite set $A_1 > y$
such that $\Kcal^{X, A_0 \times A_1}$ is $2k$-big above~$\sigma$. 
\end{definition}

The following notion of uniform density is a fairness property
of the context~$X$. It states that if the requirement $\Kcal^{X, \Gcal}$
provides infinitely many fresh values on which the graph $\Gcal$
could create a diagonalization, then at some point, $\Gcal$ will actually
produce one.

\begin{definition}
We say $\Kcal^{X, \Gcal}$ is \emph{uniformly dense} if whenever~$\Kcal^X$
is essential below some $k$-roomy condition~$(\sigma, B)$,
$\Kcal^{X, \Gcal}$ is $k$-big above~$\sigma$.
\end{definition}

\begin{definition}
We say $(\sigma, B)$ \emph{settles} $\Kcal^{X, \Gcal}$ if either
$\sigma \in \Kcal^{X, \Gcal}$ or
there exists an~$x \in \omega$ such that $\Kcal^{X, (x, +\infty)^2} \subseteq B$.
\end{definition}

Intuitively, a condition~$(\sigma, B)$ settles a requirement~$\Kcal^{X, \Gcal}$
if it forces a finite bad outcome $\sigma \in \Kcal^{X, \Gcal}$
(usually a diagonalization)
or forces an infinite bad outcome $\Kcal^{X, (x, +\infty)^2)} \subseteq B$ 
(basically forcing the partiality of a functional). This intuition
is best illustrated by the actual requirement scheme we care about.

\begin{example}\label{ex:settles-solution}
Suppose $(\sigma, B)$ settles $\Wcal_m^{X, \Gcal}$. We claim that
if $(\sigma, B)$ appears in a sequence defining a generic function~$f$,
then $\Phi_m^{X \oplus f}$ is not a solution to~$\Gcal$.
If $\sigma \in \Wcal_m^{X, \Gcal}$, then this claim was verified by Example~\ref{ex:if-meets-wn-then-diag}.
So assume that $(\sigma, B)$ settles $\Kcal^{X, \Gcal}$ via the second clause
and fix the witness~$x$. We claim that for all $(\tau, C) \leq (\sigma, B)$
and all $y > x$, $\Phi_m^{X \oplus \tau}(y)\uparrow$ or $\Phi_m^{X \oplus \tau}(y) \downarrow \neq 1$.
It follows immediately from this claim that~$\Phi_m^{X \oplus f}$ is partial or defines a finite set
and hence is not a solution to~$\Gcal$.
To prove this claim, fix $(\tau, C) \leq (\sigma, B)$.
Suppose for a contradiction that there is a $y > x$ such that $\Phi_m^{X \oplus \tau}(y) = 1$.
Then
$$
(\exists \{y, y\} \in (x, +\infty)^2) \Phi_m^{X \oplus \tau}(y) = \Phi_m^{X \oplus \tau}(y) = 1
$$
and hence $\tau \in \Kcal^{X, (x, +\infty)^2} \subseteq B \subseteq C$, contradicting
the fact that $C$ is $k$-small above~$\tau$ for some~$k$.
\end{example}

The following lemma conforts the intuition that the settling relation
is a forcing notion, and in particular is stable by condition extension.

\begin{lemma}
If $(\sigma, B)$ settles $\Kcal^{X, \Gcal}$ and $(\tau, C) \leq (\sigma, B)$, then $(\tau, C)$
settles $\Kcal^{X, \Gcal}$.
\end{lemma}
\begin{proof}
If $\sigma \in \Kcal^{X, \Gcal}$, then by closure of~$\Kcal^{X, \Gcal}$ under suffixes,
$\tau \in \Kcal^{X, \Gcal}$ and therefore $(\tau, C)$ settles $\Kcal^{X, \Gcal}$.
Suppose now that $(\sigma, B)$ settles $\Kcal^{X, \Gcal}$ via the second clause, with some witness~$x$.
$\Kcal^{X, (x, +\infty)^2} \subseteq B \subseteq C$, therefore $(\tau, C)$
settles $\Kcal^{X, \Gcal}$ with the same witness.
\end{proof}

The following theorem is the heart of the iteration forcing.
It states that, under the assumption of uniform density,
the set of conditions settling a requirement is dense.

\begin{theorem}\label{thm:uniform-to-settling}
Let $\Kcal^{X, \Gcal}$ be a uniformly dense requirement and let $(\sigma, B)$
be a condition. There is an extension $(\tau, C) \leq (\sigma, B)$ settling $\Kcal^{X, \Gcal}$.
\end{theorem}
\begin{proof}
We can suppose w.l.o.g. that $(\sigma, B)$ is $k$-roomy for some~$k \in \omega$.

First suppose that $\Kcal^{X, \Gcal}$ is essential below $(\sigma, B)$.
By definition of uniformly dense, $\Kcal^{X, \Gcal}$
is $k$-big above~$\sigma$. By the contatenation property,
there exists an extension $\tau \succeq \sigma$
which is in~$\Kcal^{X, \Gcal} \setminus B$ and such that~$B$ is $k$-small above~$\tau$.
The condition~$(\tau, B)$ is a valid extension
settling $\Kcal^{X, \Gcal}$ via the first clause.

Suppose now that $\Kcal^{X, \Gcal}$ is not essential below $(\sigma, B)$.
By definition, there exists an~$x \in \omega$ such that for every finite set $A_0 > x$,
there exists a~$y \in \omega$ such that for every finite set $A_1 > y$,
$\Kcal^{X, A_0 \times A_1}$ is $2k$-small above~$\sigma$.
Fix such~$x$. We have two cases:
\begin{itemize}
	\item Case 1: $\Kcal^{X, (x, +\infty)^2}$ is $3k$-small above~$\sigma$.
	By the smallness additivity property, $B_1 = B \cup \Kcal^{X, (x, +\infty)^2}$ is $4k$-small above~$\sigma$.
	Therefore $(\sigma, B_1)$ is an extension of $(\sigma, B)$ settling $\Kcal^{X, \Gcal}$
	via the second clause with witness~$x$.
	
	\item Case 2: $\Kcal^{X, (x, +\infty)^2}$ is $3k$-big above~$\sigma$.
	Let $T \subseteq h^{<\omega}$ a the finite tree $3k$-bushy above~$\sigma$ 
	whose leafs are in~$\Kcal^{X, (x, +\infty)^2}$.
	There are finitely many leafs in~$T$, so there exists a finite set~$A_0$
	such that $\Kcal^{X, A_0^2}$ is $3k$-big above~$\sigma$.
	By assumption, there exists some $y > A_0$ be such that for every finite set $A_1 > y$,
	$\Kcal^{X, A_0 \times A_1}$ is $2k$-small above~$\sigma$. Fix such~$y$.
	The set $\Kcal^{X, A_0 \times (y, +\infty)}$ is $2k$-small above~$\sigma$,
	otherwise, the same argument as for~$A_0$ would give a finite set~$A_1 \subset (y, +\infty)$
	such that~$\Kcal^{X, A_0 \times A_1}$ is $2k$-big above~$\sigma$, contradicting our choice of~$y$.
	So $B_1 = (B \cup \Kcal^{X, A_0 \times (y, +\infty)})$ is $3k$-small above~$\sigma$.
	By the concatenation property, there exists a string $\tau \succeq \sigma$
	in $\Kcal^{X, A_0^2}$ such that $B_1$ is $3k$-small above~$\tau$. 
	Notice that, if we denote by $\Gamma_\tau$ the set of strings extending~$\tau$,
	\begin{eqnarray*}
	\Kcal^{X, A_0 \times (y, +\infty)} \cap \Gamma_\tau 
		&=& \{ \rho \succeq \tau : (\exists F \subseteq_{fin} \Odd(A_0 \times (y, +\infty)))R_\Kcal^X(\rho, F)\}\\
		&=& \{ \rho \succeq \tau : (\exists F \subseteq_{fin} \Odd((y, +\infty)^2))R_\Kcal^X(\rho, F)\}\\
		&=& \Kcal^{X, (y, +\infty)^2} \cap \Gamma_\tau
	\end{eqnarray*}
	This equality is due to the extra property we imposed to~$R_\Kcal^X$ in Definition~\ref{def:requirement}.
	Therefore $B_2 = B_1 \cup \Kcal^{X, (y, +\infty)^2}$ is $3k$-small above~$\tau$.
	The condition~$(\tau, B_2)$ is an extension settling~$\Kcal^{X, \Gcal}$
	via the second clause with witness~$y$.
\end{itemize}
\end{proof}

\subsubsection{The construction}

At this stage of the proof, we have proven the lemmas necessary
for constructing a function~$f$ d.n.c.\ relative to~$X$
such that~$X \oplus f$ does not compute a solution to~$\Gcal$.
We shall see later that, by considering some carefully chosen requirements,
the property of uniform density of the requirements also holds at level~$X \oplus f$.

To define~$f$, let $\Kcal_0^{X, \Gcal}, \Kcal_1^{X, \Gcal} \dots$ be a list of all the requirements.
Thanks to Theorem~\ref{thm:uniform-to-settling}, we can define by induction
an infinite decreasing sequence of conditions
$c_0 = (\epsilon, B_{DNC}^X) \geq c_1 \geq \dots$ such that for each~$s \in \omega$
\begin{itemize}
	\item[1.] $|\sigma_s| \geq s$ \hfill (by the concatenation property)
	\item[2.] $c_{s+1}$ settles $\Kcal_s^{X, \Gcal}$  \hfill (by Theorem~\ref{thm:uniform-to-settling})
\end{itemize}
where $c_s$ is some $k_s$-roomy condition $(\sigma_s, B_s)$ for $s > 0$.
We define our generic function by~$f = \bigcup_s \sigma_s$.
As $B_{DNC}^X = B_0 \subseteq B_s$ for every~$s \geq 0$, $B_{DNC}^X$ is $k$-small
above every initial segment~$\sigma_s$ of~$f$ for some~$k$, and therefore
$f$ is d.n.c.\ relative to~$X$. As well, by Example~\ref{ex:settles-solution},
$X \oplus f$ does not compute a solution to~$\Gcal$. 

\subsubsection{Forcing uniform density}

We now describe the requirements forcing uniform density at the next level.
To specify a potential requirement at the next level, we need the index
for the relation $R_\Kcal^{X \oplus f}$, defining the requirement
$\Kcal^{X \oplus f, \Gcal}$. We shall denote this index by~$\Kcal$.
For each such index~$\Kcal$ and each potential condition~$(\xi, \cdot)$ of the forcing~$\Pb^{X \oplus f}$
we will associate a requirement~$\Tcal_{\Kcal, \xi, r}^X$ at the level~$X$
ensuring that~$\Kcal^{X \oplus f, \Gcal}$ will satisfy uniform density in
the case where~$(\xi, \cdot)$ is an~$r$-roomy condition.

The requirement~$\Tcal_{\Kcal, \xi, r}^X$ consists of all
strings~$\tau \in \omega^{<\omega}$ such that there exists
a finite tree $T \subseteq h^{<\omega}$ which is $r$-bushy below~$\xi$
(and whose code is bounded by $|\tau|$ to ensure~$X$-computability of~$R_{\Tcal}^X$)
such that for each leaf~$\rho$ in~$T$, $(\exists F \subseteq_{fin} \Odd(\Gcal))R_\Kcal^{X \oplus \tau}(\rho, F)$.
Note that this is where we need the fact that~$F$ is a finite subset of $\Odd(\Gcal)$
and not simply a pair in~$\Odd(\Gcal)$.

\begin{lemma}
Let $f = \bigcup_s \sigma_s$ be a generic function defined by a sequence of conditions $(\sigma_s, B_s)$.
Each requirement $\Kcal^{X \oplus f, \Gcal}$ is uniformly dense.
\end{lemma}
\begin{proof}
Fix a requirement $\Kcal^{X \oplus f, \Gcal}$ and an $r$-roomy condition~$(\xi, D)$ such that
$\Kcal^{X \oplus f}$ is essential below $(\xi, D)$. 
By property 2.\ of the construction, 
there is a condition~$(\sigma_s, B_s)$ settling~$\Tcal_{\Kcal, \xi, r}^{X, \Gcal}$.
By definition, there are two ways in which~$(\sigma_s, B_s)$ could settle~$\Tcal_{\Kcal, \xi, r}^{X, \Gcal}$.
\begin{itemize}
	\item Case 1: $\sigma_s \in \Tcal_{\Kcal, \xi, r}^{X, \Gcal}$.
	By definition, there exists a finite tree~$T \subseteq h^{<\omega}$ which is $r$-bushy
	below~$\xi$ and such that for each leaf~$\rho$ in~$T$,
	$(\exists F \subseteq_{fin} \Odd(\Gcal))R_\Kcal^{X \oplus \sigma_s}(\rho, F)$.
	In other words, $\Kcal^{X \oplus \sigma_s, \Gcal}$ is $r$-big above~$\xi$.
	Therefore, by the use property, $\Kcal^{X \oplus f, \Gcal}$ is $r$-big above~$\xi$,
	hence $\Kcal^{X \oplus f, \Gcal}$ is uniformly dense above~$(\xi, \cdot)$.

	\item Case 2: there exists an~$x \in \omega$ such that $\Tcal_{\Kcal, \xi, r}^{X, (x, +\infty)^2} \subseteq B_s$.
	Because $\Kcal^{X \oplus f}$ is essential below $(\xi, D)$, there exists a finite set $A_0 > x$
	and a finite set $A_1 > A_0$ such that~$\Kcal^{X \oplus f, A_0 \times A_1}$
	is $r$-big above~$\xi$. Let $T \subseteq h^{<\omega}$ be a finite tree $r$-bushy
	above~$\xi$ witnessing this fact. Let~$(\sigma_t, B_t)$ be an extension of~$(\sigma_s, B_s)$
	such that~$|\sigma_t|$ is greater than the code of~$T$ and~$\sigma_t$ is large enough
	to witness that for each leaf~$\rho \in T$, $(\exists F \subseteq_{fin} \Odd(\Gcal))R_\Kcal^{X \oplus \sigma_t}(\rho, F)$.
	Such~$\sigma_t$ must exist by the use property.
	By definition, $\sigma_t \in \Tcal_{\Kcal, \xi, r}^{X, (x, +\infty)^2}$,
	but $\Tcal_{\Kcal, \xi, r}^{X, (x, +\infty)^2} \subseteq B_s \subseteq B_t$,
	contradicting the fact that $B_t$ is $k$-small above~$\sigma_t$ for some~$k$.
\end{itemize}
\end{proof}

This last lemma finishes the iteration forcing as it enables
to prove that the property of uniform density is propagated to level~$X \oplus f$.

\subsection{Ground construction}

It remains to construct an infinite, computable, locally 2-colorable graph~$\Gcal$ such that
\begin{itemize}
	\item $\Gcal$ has no computable infinite homogeneous set
	\item Each requirement $\Kcal^{\emptyset, \Gcal}$ is uniformly dense
\end{itemize}

The construction is very similar to the measure-defeating argument
used in~\cite{Bienvenu2014Ramsey} for separating $\wwkl$ from $\rcolor_2$
over~$\omega$-models.
The construction is done by a finite injury priority argument. The resulting
graph will be composed of connected components of finite size.
Each strategy will put restraints on finitely many connected components, each of finite size,
therefore each strategy will have cofinitely many vertices not restrained by strategies of higher priority
to work with.
In order to make the graph computable, the edges over domain~$\{0, \dots, s\}$ must be decided before stage~$s$.
We build a graph satisfying two kind of requirements.

\subsubsection{Forcing~$\Gcal$ not to have computable solutions}

The first class of requirements $\Rcal_e$ ensures that $\Phi_e$ does not compute an infinite $\Gcal$-homogeneous set.
$$
\Rcal_e : \Phi_e \mbox{ infinite } \imp (\exists \{x, y\} \in \Odd(\Gcal)) \Phi_e(x) = \Phi_e(y) = 1
$$

The strategy for satisfying the requirement~$\Rcal_e$ consists in waiting
until~$\Phi_e$ halts on two values~$x$ and~$y$ which are \emph{not} in connected components
restrained by strategies of higher priority, and are not in the same connected components.
Assuming that each strategy acts finitely often and each connected component is of finite size,
if~$\Phi_e$ is infinite, there must be such a pair~$\{x, y\}$ appearing at some finite stage~$s$.
At this stage, the strategy picks two fresh vertices~$z_1, z_2$ which have not appeared
yet in the construction, and adds the edges $\{x, z_1\}, \{z_1, z_2\}, \{z_2, y\}$ in~$\Gcal$.
Having done that, the strategy puts a restraint to all vertices in the connected component of~$x$ and~$y$
and is declared satisfied. The edges added ensure that~$\{x, y\} \in \Odd(\Gcal)$. As the graph is computable,
once a strategy is satisfied, it is never injured.

\subsubsection{Forcing uniform density}

The second class of requirements $\Scal_{\Kcal, \sigma, k}$ states that 
if~$\Kcal$ is a potential requirement,
if $h(|\sigma|) \geq 4k$
and if for every~$x \in \omega$, there exists a finite set $A_0 > x$ such that for every $y \in \omega$,
there exists a finite set $A_1 > y$ such that~$\Kcal^{\emptyset, A_0 \times A_1}$ is $2k$-big above~$\sigma$,
then the set~$\Kcal^{\emptyset, \Gcal}$ is $k$-big above~$\sigma$.

The strategy for satisfying the requirement~$\Scal_{\Kcal, \sigma, k}$ consists in
waiting until it finds two finite sets of vertices~$A_0$ and~$A_1$
such that the vertices of~$A_0$ and vertices of~$A_1$ live in different connected components
and are not restrained by a strategy of higher priority,
together with a finite tree $T \subseteq h^{<\omega}$ which is $2k$-bushy above~$\sigma$
and whose leaves are in~$\Kcal^{\emptyset, A_0 \times A_1}$.

If such pair of sets~$A_0$ and $A_1$ is found,
the strategy takes two fresh vertices~$a$ and~$b$ greater than any vertex
which already appeared in the construction, and adds an edge between them.
It also adds edges between members of~$A_0$ and~$a$ so that they all live in the same connected
component. Now, consider two possible extensions~$G_1, G_2$ of our current graph~$G$,
in which~$G_1$ is some completion where all vertices in~$A_1$ are connected to either~$a$ or~$b$
in a way that the graph remains bipartite, and $G_2$ is obtained from~$G$
by adding for each~$y \in A_1$ an edge $\{a, y\}$ if and only if~$\{b, y\} \in G_1$
and an edge~$\{b, y\}$ if and only if~$\{a, y\} \in G_1$.
Note that~$G_2$ is also bipartite and any pair~$\{x, y\} \in A_0 \times A_1$
is homogeneous for exactly one of~$G_1$ and $G_2$.
By the smallness additivity property, either~$\Kcal^{\emptyset, G_1}$ or~$\Kcal^{\emptyset, G_2}$
is~$k$-big above~$\sigma$, so we can complete the graph~$\Gcal$
so that~$\Kcal^{\emptyset, \Gcal}$ is $k$-big above~$\sigma$.

If no such pair of sets exists, we claim that the requirement is vacuously satisfied.
Suppose it is not and assume that
each connected component is of finite size, and there are finitely many connected components
restrained by strategies of higher priority. Taking~$x$
to be greater than all restrained vertices, there must be a finite set~$A_0 > x$
such that for every~$y$, there exists a finite set~$A_1 > y$ 
such that~$\Kcal^{\emptyset, A_0 \times A_1}$ is $2k$-big above~$\sigma$.

By choice of~$x$, members of~$A_0$ are not restrained by a strategy of higher priority.
The connected components of members of~$A_0$ being finite, taking~$y$ to be greater
than all of those connected components, there must exist a finite set $A_1 > y$
such that~$\Kcal^{\emptyset, A_0 \times A_1}$ is $2k$-big above~$\sigma$.
By choice of~$y$, members of~$A_1$ are in different connected components
than members of~$A_0$ and are not restrained by strategies of higher priority.
By definition of~$2k$-bigness, there exists a finite tree~$T$ which is~$2k$-bushy above~$\sigma$
and whose leaves are in~$\Kcal^{\emptyset, A_0 \times A_1}$. Therefore we are in first case.

\subsubsection{The construction}

The two kinds of requirements are interleaved into a priority ordering.
Notice that as soon as an edge is put between two vertices~$x$ and~$y$,
the corresponding strategy puts a restraint on the whole connected component,
so that no strategy of lower priority can add edges to it.
As each strategy adds finitely many vertices, each connected component
of the resulting graph~$\Gcal$ must be of finite size.
As well, no edge is added between two vertices~$x, y < s$ after stage~$s$,
therefore the graph is computable.

This finishes the ground construction.

\subsection{Putting all together}

Thanks to the ground construction and the iteration forcing,
we can define an infinite increasing sequence of sets~$X_0 = \emptyset \leq_T X_1 \leq_T \dots$
such that for each~$i \in \omega$
\begin{itemize}
	\item $X_{i+1}$ computes a function d.n.c.\ relative to~$X_i$
	\item $X_i$ does not compute a solution to~$\Gcal$
\end{itemize}

Let~$\Mcal$ be $\omega$-structure whose second-order part is the Turing ideal
generated by the sequence~$X_0, X_1, \dots$
We first check that~$\Mcal$ is a model of~$\dnrs{h}$.
Fix a set~$Y \in \Mcal$. There exists an~$i$ such that~$Y \leq_T X_i$,
and so such that~$X_{i+1}$ computes an $h$-bounded function~$f$ d.n.c.\ relative to~$Y$.
By construction of~$\Mcal$, $X_{i+1} \in \Mcal$ and so~$f \in \Mcal$.
Therefore~$\Mcal$ is a model of~$\dnrs{h}$.
There remains to check that~$\Mcal$ is not a model of~$\rcolor_2$.
As~$\Gcal$ is computable, $\Gcal \in \Mcal$.
For every set~$Y \in \Mcal$, there exists an~$i$ such that $Y \leq_T X_i$
and therefore~$Y$ is not a solution to~$\Gcal$. Hence~$\Gcal$
has no solution in~$\Mcal$ and thus is not a model of~$\rcolor_2$.

%% file: doc.bbl
\vspace{0.5cm}

\begin{thebibliography}{10}

\bibitem{Ambos-Spies2004Comparing}
Klaus Ambos-Spies, Bj{\o}rn Kjos-Hanssen, Steffen Lempp, and Theodore~A Slaman.
\newblock {Comparing DNR and WWKL}.
\newblock {\em Journal of Symbolic Logic}, 69(04):1089--1104, 2004.

\bibitem{BienvenuEvery}
Laurent Bienvenu and Ludovic Patey.
\newblock {Diagonally non-computable functions and fireworks}.
\newblock submitted.

\bibitem{Bienvenu2014Ramsey}
Laurent Bienvenu, Ludovic Patey, and Paul Shafer.
\newblock {A Ramsey-type K{\"{o}}nig's lemma and its variants}.
\newblock In preparation, 2014.

\bibitem{Dorais2014Comparing}
Fran{\c{c}}ois~G Dorais, Jeffry~L Hirst, and Paul Shafer.
\newblock {Comparing the strength of diagonally non-recursive functions in the
  absence of $\Sigma^0_2$ induction}.
\newblock {\em arXiv preprint arXiv:1401.3823}, 2014.

\bibitem{Flood2012Reverse}
Stephen Flood.
\newblock Reverse mathematics and a {R}amsey-type {K\"o}nig's lemma.
\newblock {\em Journal of Symbolic Logic}, 77(4):1272--1280, 2012.

\bibitem{Flood2014Separating}
Stephen Flood and Henry Towsner.
\newblock {Separating principles below WKL${}_0$}, 2014.
\newblock in preparation.

\bibitem{Friedman1974Some}
Harvey Friedman.
\newblock Some systems of second order arithmetic and their use.
\newblock In {\em Proceedings of the International Congress of Mathematicians,
  Vancouver}, volume~1, pages 235--242, 1974.

\bibitem{Greenberg2009Lowness}
Noam Greenberg and Joseph Miller.
\newblock Lowness for {K}urtz randomness.
\newblock {\em Journal of Symbolic Logic}, 74(2):665--678, 2009.

\bibitem{Hirst1990Marriage}
Jeffry~L. Hirst.
\newblock Marriage theorems and reverse mathematics.
\newblock In Wilfried Sieg, editor, {\em Logic and Computation}, volume 106 of
  {\em Contemporary Mathematics}, pages 181--196. American Mathematical
  Society, 1990.

\bibitem{Jockusch1972Degrees}
C~Jockusch and R~Soare.
\newblock Degrees of members of {$\Pi^0_1$} classes.
\newblock {\em Pacific Journal of Mathematics}, 40:605--616, 1972.

\bibitem{Jockusch1989Recursively}
Carl~G Jockusch, Manuel Lerman, Robert~I. Soare, and Robert~M Solovay.
\newblock {Recursively enumerable sets modulo iterated jumps and extensions of
  Arslanov's completeness criterion}.
\newblock {\em The Journal of Symbolic Logic}, 54(4):1288--1323, 1989.

\bibitem{Jockusch2013Diagonally}
Carl~G Jockusch and Andrew~EM Lewis.
\newblock Diagonally non-computable functions and bi-immunity.
\newblock {\em Journal of Symbolic Logic}, 78(3):977--988, 2013.

\bibitem{JockuschJr1989Degrees}
Carl~G Jockusch~Jr.
\newblock Degrees of functions with no fixed points.
\newblock {\em Studies in Logic and the Foundations of Mathematics},
  126:191–201, 1989.

\bibitem{Kjos-Hanssen2009Infinite}
Bj{\o}rn Kjos-Hanssen.
\newblock Infinite subsets of random sets of integers.
\newblock {\em Mathematics Research Letters}, 16:103--110, 2009.

\bibitem{Kuvcera1985Measure}
Antonin Ku{\v c}era.
\newblock Measure, {$\Pi^0_1$} classes, and complete extensions of {PA}.
\newblock {\em Lecture Notes in Mathematics}, 1141:245--259, 1985.

\bibitem{kumabe2009fixed}
Masahiro Kumabe and Andrew~EM Lewis.
\newblock A fixed point free minimal degree.
\newblock {\em Journal of the London Mathematical Society}, 80(3):785--797,
  2009.

\bibitem{Lerman2013Separating}
Manuel Lerman, Reed Solomon, and Henry Towsner.
\newblock Separating principles below ramsey's theorem for pairs.
\newblock {\em Journal of Mathematical Logic}, 13(02):1350007, 2013.

\bibitem{Liu2012RT22}
Jiayi Liu et~al.
\newblock {RT$^2_2$ does not imply WKL$_0$}.
\newblock {\em Journal of Symbolic Logic}, 77(2):609--620, 2012.

\bibitem{MillerPersonal}
Joe Miller.
\newblock Personal communication.

\bibitem{MushfeqKhan2014Forcing}
Joseph~Miller Mushfeq~Khan.
\newblock {Forcing with Bushy Trees}.
\newblock preprint, 2014.

\end{thebibliography}
